\documentclass[a4paper]{article}

\usepackage[english]{babel}

\usepackage[utf8]{inputenc}
\usepackage[utf8]{inputenc}
\usepackage{amsmath}
\usepackage{graphicx}
\usepackage{amssymb}
\usepackage{amsthm}
\usepackage{tikz-cd}
\usepackage{mathrsfs}
\usepackage[colorinlistoftodos]{todonotes}
\usepackage{enumitem}
\usepackage{yfonts}
\usepackage{ dsfont }
\usepackage{MnSymbol}
\usepackage{slashed}

\title{Metric SYZ conjecture for certain toric Fano hypersurfaces}

\author{Yang Li}

\date{\today}
\newtheorem{thm}{Theorem}[section]
\newtheorem{lem}[thm]{Lemma}

\theoremstyle{definition}
\newtheorem{eg}[thm]{Example}

\newtheorem{conj}[thm]{Conjecture}
\newtheorem{cor}[thm]{Corollary}
\newtheorem{claim}[thm]{Claim}
\newtheorem{rmk}[thm]{Remark}
\newtheorem{prop}[thm]{Proposition}
\newtheorem{Def}[thm]{Definition}

\newtheorem*{Acknowledgement}{Acknowledgement}

\newcommand{\cf}{\emph{cf.} }

\newcommand{\R}{\mathbb{R}}
\newcommand{\C}{\mathbb{C}}
\newcommand{\Z}{\mathbb{Z}}

\newcommand{\Q}{\mathbb{Q}}

\newcommand{\norm}[1]{\left\lVert#1\right\rVert}

\def\XXint#1#2#3{{\setbox0=\hbox{$#1{#2#3}{\int}$ }
		\vcenter{\hbox{$#2#3$ }}\kern-.6\wd0}}

\begin{document}
	\maketitle

\begin{abstract}
We prove the metric version of the SYZ conjecture for a class of Calabi-Yau hypersurfaces inside toric Fano manifolds, by solving a variational problem whose minimizer may be interpreted as a global solution of the real Monge-Amp\`ere equation on certain polytopes. This does not rely on discrete symmetry.
\end{abstract}

\section{Introduction}

The metric aspect of the Strominger-Yau-Zaslow  (SYZ) conjecture \cite{SYZ} asks for the following:

\begin{conj}\label{SYZconj}(\cf \cite{LiNA}\cite{Lisurvey}\cite{SYZ})
Let $X_t$ be a 1-parameter maximally degenerate family of polarized $d$-dimensional  Calabi-Yau  manifolds over the punctured disc $\mathbb{D}_t^*$, then for any given $0<\delta\ll 1$, and all $t$ small enough depending on $\delta$, there exists a special Lagrangian $T^d$-fibration on an open subset of $X_t$ for $0<|t|\ll 1$, whose normalized Calabi-Yau measure is at least $1-\delta$.
\end{conj}

The case of Abelian varieties is classical, while the case of K3 surfaces is known through the trick of hyperk\"ahler rotation \cite{GrossWilson}\cite{Odaka}. The first nontrivial example in all dimensions is the Fermat family of Calabi-Yau hypersurfaces \cite{LiFermat}
\begin{equation*}%\label{Fermat}
X_s= \{ Z_0\ldots Z_{d+1}+ e^{-s} \sum_0^{d+1} Z_i^{d+2}=0   \}\subset \mathbb{CP}^{d+1}, \quad s\gg 1.
\end{equation*}
The proof of the SYZ conjecture in this case exploits the large discrete symmetry group, to simplify some rather delicate combinatorial problems.

A turn of thinking came about in \cite{LiNA}, where we reduced the SYZ conjecture to a problem of algebraic geometry. In brief, there is a non-archimedean  (NA) analogue of the complex Monge-Amp\`ere (MA) equation, defined in terms of intersection theory, and the deep work of Boucksom et al. \cite{Boucksom}\cite{Boucksom1}\cite{Boucksomsemipositive}\cite{Boucksomsurvey} proved the existence and uniqueness of the solution. To deduce the SYZ conjecture, what is left to prove is a `comparison property' \cite{LiNA}, which morally asserts the NA MA solution is not too wild.

For this new strategy to be convincing, it is requisite to verify the comparison property in  examples. The classical case of Abelian varieties was checked by Goto-Odaka \cite{GotoOdaka}. The recent work of Pille-Schneider \cite{PilleSchneider} and Hultgren et al. \cite{Hultgren} independently verified the comparison property for some generalisation of the Fermat family, but both works still rely heavily on the large discrete symmetry group, and thus still only apply to certain hypersurfaces inside $\mathbb{CP}^{d+1}$.

A notable contribution in \cite{Hultgren} is to introduce a global variational problem. The existence and uniqueness of the minimizer is quite transparent, but the local PDE nature of the minimizer is not manifest.
It is essentially known in \cite{Hultgren} that if one can deduce a local real Monge-Amp\`ere equation from the minimizer, then it would naturally produce the unique solution of the NA MA solution, and the `comparison property' would be a corollary. The present paper aims to make progress on this strategy, by proving the expected properties of the minimizer for more examples without explicit reliance on the discrete symmetry.

Let $\Delta\subset M_\R$ be an integral reflexive Delzant polytope, and $X_\Delta$ be the associated smooth toric Fano manifold of dimension $d+1$, with an ample polarization $L\to X_\Delta$, which needs not be the anticanonical polarization.
The origin $0\in \Delta$ corresponds to a distinguished section $X_{can}\in H^0(X_\Delta, -K_{X_\Delta})$, which defines the toric boundary of $X_\Delta$. Let $F\in H^0(X_\Delta, -K_{X_{\Delta}})$ be a generic section, and by assumption the divisor $\{ F=0 \}\subset X_\Delta$ is smooth, and intersects all the toric boundary strata transversely, and in particular does not pass through the finite number of toric fixed points on $X_\Delta$. We will consider the family of Calabi-Yau hypersurfaces as $t\to 0$:
\begin{equation}\label{torichypersurface}
X_t= \{ X_{can}+tF=0  \}\subset X_\Delta\times \C^*_t.
\end{equation}
Algebro-geometrically $X_t$ degenerate to the toric boundary.

We will assume one extra property on the reflexive polytope $\Delta$, whose dual polytope is denoted as $\Delta^\vee$:
\begin{equation}\label{extraassumption}
\langle m, n\rangle \neq 0,\quad \forall m\in \text{vertex}(\Delta),\quad \forall n\in \text{vertex}(\Delta^\vee).
\end{equation}
The main outcome is 
\begin{thm}\label{SYZthm}
	The SYZ conjecture \ref{SYZconj} holds for the family $X_t$ as $t\to 0$.
\end{thm}

The main new content of this paper is contained in section \ref{VariationalProblem}, which establishes a structure theory for the minimizer of the variational problem, and has the flavour of convex analysis, variational calculus, and real MA equation. This part does not require prior knowledge of K\"ahler geometry, and may be of independent interest to analysts. In section \ref{applicationtoSYZ} we briefly sketch the complex geometric aspect towards the application to the SYZ conjecture, relying on the aforementioned results in \cite{LiFermat}\cite{LiuniformSkoda}\cite{LiNA}\cite{Hultgren}.

\begin{rmk}
J. Hultgren informs the author that together with R. Andreasson they have independently found an equivalent condition in terms of optimal transport theory, for the minimizer of the variational problem to admit the expected PDE interpretation \cite{Hultgrennew}. Their upcoming paper will also contain counterexamples, demonstrating in particular that some nontrivial condition on the reflexive polytope is necessary in this strategy.
\end{rmk}

\begin{Acknowledgement}
The author is a current Clay Research Fellow based at MIT. He would like to especially thank J. Hultgren for pointing out his upcoming work, and for helpful comments on this manuscript. He also thanks S. Boucksom, L. Pille-Schneider, and S. Sun for discussions in the past.
\end{Acknowledgement}

\section{The variational problem}\label{VariationalProblem}

\subsection{Motivation: A global real Monge-Amp\`ere equation?}

Let $M_\R\simeq \Z^{d+1}\otimes_\Z \R$ and $N_\R=M_\R^\vee$ be a pair of dual vector spaces, and 
let $\Delta\subset  M_\R$ be an integral reflexive polytope, with dual polytope $\Delta^\vee\subset N_\R$, so that
\[
\Delta^\vee= \{ x\in N_\R| \langle m, x\rangle \leq 1, \forall m\in \text{vertex}(\Delta)      \},
\] 
\[
\Delta= \{ p\in M_\R| \langle n, p \rangle \leq 1, \forall n\in \text{vertex}(\Delta^\vee)      \},
\]
both $\Delta, \Delta^\vee$ contain the origin, and all the vertices are integral points. Given negative real numbers $\mu(n)$ for all $n\in  \text{vertex}(\Delta^\vee)$, and $\lambda(m)$ for all $m\in \text{vertex}(\Delta)$, we can define the convex functions on $N_\R$ and $M_\R$,
\[
L_\lambda(x)=  \max_m \langle x, m\rangle+  \lambda(m) ,\quad L_\mu(p)= \max_n \langle p, n\rangle+  \mu(n) .
\]
These in turn define the convex polytopes containing the origin in the interior,
\[
\Delta_\lambda^\vee= \{ L_\lambda(x)\leq 0   \}\subset N_\R,\quad  \Delta_\mu= \{ L_\mu(p)\leq 0   \}\subset M_\R.
\]
For instance, the choice $\lambda=-1$ recovers $\Delta^\vee$, and the choice $\mu=-1$ recovers $\Delta$.

The normal vectors to the top dimensional faces are labelled by the vertices $m,n$, which are primitive integral vectors. This integral structure on the faces induce a canonical Lebesgue measure $dx$ and $dp$ on $\partial \Delta_\lambda^\vee$ and $\partial \Delta_\mu$ respectively, which we normalize to probability measures 
\begin{equation}\label{normalizedmeasure}
d\mathcal{L}^\vee= \frac{1}{  \int_{\partial \Delta_\lambda^\vee} dx  }dx,\quad d\mathcal{L}= \frac{1}{  \int_{\partial \Delta_\mu}dp  }dp.
\end{equation}
We shall use the notation $|E|$ to denote the measure of subsets $E$ with respect to $d\mathcal{L}$ or $d\mathcal{L}^\vee$.

Our secret goal is to formulate and solve a version of the real MA equation on $\partial \Delta_\lambda^\vee$ and $\partial \Delta_\mu$. A direct attempt faces the following difficulties:
\begin{itemize}
\item  The polytopes do not have natural global affine structures. On the overlap of the candidate charts, the transition functions will only be piecewise affine.

\item  Consequently, the condition of being a convex function depends on local charts.

\item  The definition of the real MA equation via local charts will not be compatible with piecewise affine transition functions.

\end{itemize}

An insight of Hultgren et al. \cite{Hultgren}, carried out in a very symmetric special case, is to adopt instead a \emph{global variational formulation}. The price is that the PDE nature of the problem is not manifest.

\subsection{Legendre transform on polytopes}\label{Legendretransformonpolytopes}

We define the \textbf{Legendre transforms} (referred to as `c-transforms' in \cite{Hultgren}, which comes from optimal transport nomenclature),
\[
L^\infty(\partial \Delta_\mu)\to L^\infty(\partial \Delta_\lambda^\vee),\quad  L^\infty(\partial \Delta_\lambda^\vee)\to L^\infty(\partial \Delta_\mu),
\]
such that for $\phi\in L^\infty(\partial \Delta_\lambda^\vee)$ and $\psi\in L^\infty(\partial \Delta_\mu)$,
\begin{equation}\label{Legendre}
\psi^*(x)= \sup_{p\in \partial\Delta_\mu}\langle x,p\rangle - \psi(p),\quad
\phi^*(p)= \sup_{x\in \partial\Delta_\lambda^\vee} \langle x, p\rangle -\phi(x).
\end{equation}
The theory works symmetrically for $\phi,\psi$.

Some of the immediate formal properties are listed below (\cf \cite[section 3]{Hultgren}):

\begin{lem}\label{EquicontinuityLegendre}
(Equicontinuity) The Legendre transforms $\phi^*, \psi^*$ have uniformly bounded Lipschitz constant.
\end{lem}

\begin{lem}\label{InvolutionLegendre}
(Involution) For arbitrary $\phi\in L^\infty(\partial \Delta_\lambda^\vee)$, the double Legendre transform $\phi^{**}\leq \phi$. If $\phi=\psi^*$ for some $\psi$, then $\phi^{**}=\phi$. 
\end{lem}

\begin{lem}\label{MonotonicityLegendre}
(Monotonicity)
If $\phi_1\geq \phi_2$, then $\phi_1^{*}\leq \phi_2^*$. For any constant $c\in \R$, we have $(\phi+c)^*=\phi^*-c$. Morever $\norm{\phi_1^*-\phi_2^*}_{C^0}\leq \norm{\phi_1-\phi_2}_{C^0}$.
\end{lem}

We then introduce a class of functions $\mathcal{P}\subset C^0(\partial \Delta_\lambda^\vee )$ and $\mathcal{P}^\vee\subset C^0(\partial\Delta_\mu)$,  defined as the images of $L^\infty$ under the Legendre transform. By the involution property, the Legendre transform sets up a canonical isomorphism $\mathcal{P}\simeq \mathcal{P}^\vee$, which  is isometric with respect to the $C^0$-norms by Lemma \ref{MonotonicityLegendre}. The  formula
\[
\phi(x)=\max_{p\in \partial \Delta_\mu } \langle x,p\rangle -\phi^*(p),\quad \phi^*(p)=\max_{x\in \partial \Delta_\lambda^\vee} \langle x,p\rangle- \phi(x)
\]
provides a canonical \emph{extension} of functions in $\mathcal{P},\mathcal{P}^\vee$ to convex functions on the ambient Euclidean spaces $N_\R$ and $M_\R$. We can thus think of $\mathcal{P}, \mathcal{P}^\vee$ as `\emph{global convex functions}', which in particular have gradient contained in $\Delta_\mu, \Delta_\lambda^\vee$ respectively.

\subsection{Conjugate sets and gradients}\label{conjugategradient}

We begin with a classical analogy. In the classical Legendre transform theory for convex functions on $\R^n$, for a given position variable $x$, the gradient $\nabla \phi(x)$ is recovered by the conjugate momentum $p$ achieving $\phi(x)=\langle x,p\rangle-\phi^*(p)$. In the global context of polytopes, finding a reasonable analogue of the conjugate $p$ to $x$ is a rather delicate question, and there are at least two natural notions.

\begin{Def}
Given a function $\phi\in \mathcal{P}$, 
let $x\in \partial \Delta_\lambda^\vee$, then the \emph{conjugate set} is 
\[
\bar{\nabla} \phi(x):=\{   p\in \partial\Delta_\mu | \phi^*(p)+ \phi(x)= \langle x,p\rangle    \}.
\] 
For $E\subset \partial \Delta_\lambda^\vee$, the conjugate set is
\[
\bar{\nabla} \phi(E):=\{   p\in \partial\Delta_\mu | \phi^*(p)+ \phi(x)= \langle x,p\rangle,  \text{ for some } x\in E    \}.
\]
We leave the reader to define the Legendre dual version $\bar{\nabla} \phi^*$. 
\end{Def}

In the polytope setting, there is a more subtle new notion, which better captures the classical intuition of gradient. To motivate this, we first delve a little into polyhedral geometry. The top dimensional faces $\Delta_m^\vee$ of $\partial \Delta_\lambda^\vee$ are labelled by the vertices $m$ of $\Delta$, and concretely
\[
\Delta_m^\vee=\{ \langle x, m\rangle =-\lambda(m)  \}\subset \partial \Delta_\lambda^\vee.
\]
Likewise the vertices $n$ of $\Delta^\vee$ label the faces $\Delta_n\subset \partial \Delta_\mu$.

Given $\phi\in \mathcal{P}$, we can restrict $\phi$ to a convex function on $\text{Int}(\Delta_m^\vee)$, which is just a  convex domain in some Euclidean hyperplane. The classical gradient of this restricted function naturally lies in the quotient space $M_\R/\R m$, and a moment of thought reveals the gradient is contained in the image of $\Delta_\mu$ under the quotient map $M_\R\to M_\R/\R m$.

\begin{lem}\label{Hm+-}
Each fibre under the quotient map \[\Delta_\mu\subset M_\R\to \text{Image}(\Delta_\mu)\subset M_\R/\R m\] intersects each of the following two unions of faces $H_m^+,H_m^-\subset\partial \Delta_\mu$ uniquely at one point:
\[
H_m^+= \bigcup \{ \Delta_n  : \text{vertex $n\in  \partial \Delta^\vee$ with $\langle n, m\rangle =1$}   \} ,
\]
\[
H_m^-= \bigcup \{ \Delta_n  : \text{vertex $n\in  \partial \Delta^\vee$ with $\langle n, m\rangle \leq -1$}   \} ,
\]
In particular $H_m^+, H_m^-$ each projects homeomorphically onto $\text{Image}(\Delta_\mu)\subset M_\R/\R m$.	

\end{lem}

\begin{rmk}
We leave the reader to write down the Legendre dual statement involving the self explanatory notations $H_n^+, H_n^-$.
\end{rmk}

\begin{proof}
For every point $p\in \Delta_\mu$, the line $p+\R m$ intersects the convex polytope $\Delta_\mu$ in a line segment, with two endpoints (unless we are in the degenerate setting where the line segment becomes a point). Now the vertices $n$ prescribe the normal vectors to the faces. In order for the line segment to exit $\Delta_\mu$ in the positive (resp. negative) $m$ direction, we need $\langle n,m\rangle >0$ (resp. $<0$). By assumption, all $n,m$ are integral vectors, so $\langle m,n\rangle$ is an integer. Furthermore by assumption $\langle m,n\rangle \leq 1$ for any choice of vertices $m,n$.
The result follows.
\end{proof}

It is desirable that the gradient $p$ at $x$ is contained in $H_m^+$ so that it has a unique natural identification with a point in $M_\R/\R m$.

\begin{Def}
Let $x\in \partial \Delta_\lambda^\vee$, which may lie on possibly several faces. Then the \emph{gradient set} of $x$ is the intersection
\[
\nabla \phi (x)= \bar{\nabla}\phi(x)\cap \bigcap_m \{  H_m^+ : x\in \Delta_m^\vee  \}.
\]
We say $x$ has \emph{anomalous conjugate points} if $\bar{\nabla}\phi(x)\setminus \nabla\phi(x)$ is nonempty.
Likewise we define $\nabla\phi(E)$ for subsets $E$, and the Legendre dual versions etc.
\end{Def}

\begin{rmk}
If $p\in \bar{\nabla}\phi(x)$ is a conjugate point, then tautologically $x\in \bar{\nabla}\phi^*(p)$. However $p\in \nabla\phi(x)$ does not automatically imply $x\in \nabla\phi^*(p)$; one sufficient condition is that $p$ lies on only one face.

\end{rmk}

While conjugate points exist quite obviously by the definition of the Legendre transform, the existence of at least one gradient $p$ for any given $x$ requires a little more thought:

\begin{lem}\label{gradientexistencelem}
If $x\in \partial \Delta_\lambda^\vee$ has an anomalous conjugate point $p'$, then it also admits a gradient $p\in \nabla\phi(x)$ (so in particular gradients always exist).
\end{lem}

\begin{proof}
Suppose $x$ lies on the intersection of some faces $\Delta_{m_i}^\vee$. By the definition of conjugate point
$
\phi^*(p')+\phi(x)= \langle x,p\rangle.
$
Consider the polyhedral set $ (p'+\sum_i \R_{\geq 0} m_i )\cap \Delta_\mu$, which must have some extremal point $p$ such that $
 (p+\sum_i \R_{\geq 0} m_i )\cap \Delta_\mu
$
consists of only one point $p$. Then for each $m_i$, there must be some $n_i$ with $\langle m_i, n_i\rangle>0$ such that $p\in \Delta_{n_i}$, hence $p\in \bigcap H_{m_i}^+$.

We claim $p$ is also a conjugate point of $x$. We write $p=p'+\sum s_i m_i$ with $s_i\geq 0$. 
Now since $\phi^*$ as a function on $M_\R$ has gradient contained inside $\Delta_\lambda^\vee$, we have
\[
\phi^*(p)-\phi^*(p') \leq \max_{y\in \Delta_\lambda^\vee} \langle p-p', y\rangle \leq \sum s_i \max_{y\in \Delta_\lambda^\vee} \langle m_i, y\rangle \leq -\sum s_i \lambda(m_i) .
\]
Hence
\[
\phi^*(p)+ \phi(x) \leq \phi^*(p')+\phi(x) - \sum s_i\lambda(m_i)=\langle x, p'\rangle -\sum s_i\lambda(m_i)= \langle x, p'+\sum s_im_i \rangle = \langle x, p\rangle. 
\]
This combined with the tautological inequality $\phi^*(p)+ \phi(x)\geq  \langle x, p\rangle$ shows that the equalities hold, so $p$ is a conjugate point, hence a gradient.
\end{proof}

\begin{rmk}
	The arguments above will not work if we attempt to define the gradient to lie in $H_m^-$ instead of $H_m^+$. In fact, very often there is no conjugate point in $H_m^-$.
\end{rmk}

\begin{cor}\label{anomalousrigidity}
Let $p\in \bar{\nabla}\phi(x)$, such that there is an anomalous conjugate point $p'$ with $p=p'+\sum s_im_i$ and $s_i\geq 0$ as above. Then for each $m_i$ with $x\in \Delta_{m_i}^\vee$, either $s_i=0$ or $\bar{\nabla}\phi^*(p)\subset \Delta_{m_i}^\vee$.

\end{cor}

\begin{proof}
Since the equalities are achieved in the proof of Lemma \ref{gradientexistencelem}, we have $\phi^*(p)-\phi^*(p')=-\sum s_i\lambda(m_i)$. For any $y\in \partial\Delta_\lambda^\vee$,  we have
\[
\begin{split}
\phi(y)\geq \langle y, p'\rangle -\phi^*(p') =\langle y, p'\rangle -\phi^*(p)-\sum s_i\lambda(m_i)
\\
= \langle y, p\rangle -\phi^*(p)-\sum s_i(\lambda(m_i)+ \langle y, m_i\rangle).
\end{split}
\]
Now suppose $y$ is any conjugate point of $p$, then $\phi(y)=\langle y, p\rangle-\phi^*(p)$, whence
\[
\sum s_i (\lambda(m_i)+ \langle y, m_i\rangle)\geq 0. 
\]
By construction $\lambda(m_i)+ \langle y, m_i\rangle\leq 0 $ for any $y\in \Delta_\lambda^\vee$.
For any $i$, as long as $s_i>0$, then $\langle y, m_i\rangle =-\lambda(m_i)$, namely $y\in \Delta_{m_i}^\vee$.
\end{proof}

\begin{rmk}\label{anomalousrigidityrmk}
In the special case that $x$ lies on the interior of one face, then there is only one $m=m_i$ involved, with $p=p'+sm$, and since $p\neq p'$, we are forced to have $s>0$, and $\bar{\nabla}\phi^*(p)\subset \Delta_{m}^\vee$. 	
\end{rmk}

\begin{prop}
(Comparison of gradient notions)  Let $x\in \text{Int}(\Delta_m^\vee)$, then $p\in H_m^+$ is a gradient in $\nabla\phi(x)$, if and only if
$p\in  M_\R/\R m$ 
 under the identification $H_m^+\simeq \text{Image}(\Delta_\mu)\subset M_\R/\R m$
 is a classical gradient of the convex function $\phi$ on $\text{Int}(\Delta_m^\vee)$.

\end{prop}

\begin{proof}
If $p\in \nabla\phi(x)$, then restricted to $\text{Int}(\Delta_m^\vee)$ we have
\[
\phi(y)\geq \langle y,p\rangle- \phi^*(p)= \phi(x)+ \langle y-x, p\rangle,\quad \forall y\in\text{Int}(\Delta_m^\vee) ,
\]
which means $p$ is a classical gradient.

For the converse, we suppose $p\in M_\R/\R m$ is a classical gradient at $x$, which is identified uniquely with $p\in \Delta_n\subset H_m^+$. 
Since $\Delta_\lambda^\vee$ lies in the half space $\langle m, x\rangle+\lambda(m)\leq 0$, and $\langle n,m\rangle=1$, we can write any $y\in \partial \Delta_\lambda^\vee$ as 
\[
y=y'-sn,\quad s\geq 0, \quad \langle m, y'\rangle+\lambda(m)= 0.
\] 
Recall $\phi$ has a canonical extension to a convex function on $N_\R$. Restricted to the hyperplane $ \langle m,\cdot\rangle+\lambda(m)= 0$, the classical gradient property implies
\[
\phi(y') \geq \phi(x)+ \langle y'-x, p\rangle.
\]
But since the convex function $\phi$ on $N_\R$ has gradient contained in $\Delta_\mu$,
\[
\phi(y') \leq \phi(y) + \max_{ \Delta_\mu }\langle y'-y, \cdot\rangle \leq \phi(y) -s\mu(n).
\]
Combining the above,
\[
\phi(y)\geq \phi(x)+ \langle y'-x, p\rangle +s\mu(n)= \phi(x)+ \langle y'-x-sn, p\rangle =\phi(x)+ \langle y-x, p\rangle.
\]
Since $y$ is arbitrary, we deduce $\phi(x)+ \phi^*(p)=\langle x, p\rangle$, namely $p\in \bar{\nabla}\phi(x)$. Since $x$ lies on only one face, and $p\in H_m^+$, we conclude $p\in \nabla\phi(x)$.
\end{proof}

\subsection{Variational problem}

We now set up a global functional on $\mathcal{P}\simeq \mathcal{P}^\vee$:
\begin{equation}\label{functional}
\mathcal{F}(\phi)= \int_{\partial \Delta_\lambda^\vee} \phi d\mathcal{L}^\vee +  \int_{\partial \Delta_\mu} \phi^* d\mathcal{L}.
\end{equation}
This functional is manifestly symmetric in terms of the Legendre transform; this `mirror symmetry' is fundamental to our approach. In this variational problem, the \emph{existence of a minimizer} is quite transparent, but the local PDE nature of the minimizer is less obvious, since the definition of the Legendre transform is not local, and the differentiability of the functional at the critical point is not a priori clear.

\begin{thm}\label{existence}
	The functional admits a minimizer in $\mathcal{P}$, henceforth denoted as $\phi$.
\end{thm}

\begin{proof}(\cf \cite[Thm 5.2]{Hultgren})
By Lemma \ref{MonotonicityLegendre} and the fact that $d\mathcal{L}, d\mathcal{L}^\vee$ are both probability measures, for any $c\in \R$ we have $\mathcal{F}(\phi+c)=\mathcal{F}(\phi)$. This allows us to impose without loss that $\max_{\partial\Delta_\lambda^\vee} \phi=0$. Then by the uniform Lipschitz estimate in $\mathcal{P}$, we can envoke Arzela-Ascoli to take a $C^0$ limit of a minimizing sequence of the functional. By Lemma \ref{MonotonicityLegendre}, the Legendre transformed functions also converge in $C^0$, so the limit attains the minimum.
\end{proof}

\begin{thm}\label{uniqueness}
The minimizer is unique up to an additive constant.
\end{thm}

\begin{proof}
(compare \cite[Thm 5.2]{Hultgren})	Suppose $\tilde{\phi}$ is another minimizer, and we consider $\phi'=\frac{1}{2}(\phi+\tilde{\phi})$, which still lies in $\mathcal{P}$ by the convexity of the function class $\mathcal{P}$. Now for any $p\in \partial\Delta_\mu$,
\[
\begin{split}
& \phi'^{*}(p)=\max \langle x, p\rangle- \phi'(x)
\\
\leq  &\frac{1}{2}\max (\langle x, p\rangle- \phi(x))+\frac{1}{2}\max (\langle x, p\rangle- \tilde{\phi}(x))= \frac{1}{2}(\phi^*(p)+ \tilde{\phi}^*(p)),
\end{split}
\]
whence $\mathcal{F}(\phi') \leq \frac{1}{2}(\mathcal{F}(\phi)+ \mathcal{F}(\tilde{\phi} ) )$. This forces $\phi'$ to be also a minimizer, and the equality is achieved. If $p\in \nabla \phi'(x)$, then  the equality $\phi'^{*}(p)= \frac{1}{2}(\phi^*(p)+\tilde{\phi}^*(p))$ implies
\[
(\phi(x)+ \phi^*(p) -\langle x,p\rangle) + (\tilde{\phi}(x) +\tilde{\phi}^*(p)-\langle x,p\rangle)=0,
\]
which forces $p\in \nabla \phi(x)$ and $p\in \nabla \tilde{\phi}(x)$.

When we restrict to any open face $\text{Int}(\Delta_m^\vee)$ of $\partial \Delta_\lambda^\vee$, the functions $\phi, \tilde{\phi}$ are convex functions on convex domains, and $\nabla\phi(x), \nabla\tilde{\phi}(x)$ are identified with the classical gradients of convex functions, which naturally lie inside the quotient space $M_\R/\R m$ (\cf section \ref{conjugategradient}). Since $\nabla \phi= \nabla \tilde{\phi}$ Lebesgue-a.e, the difference $\phi-\tilde{\phi}$ is a constant on the face (\cf \cite[Lem 5.3]{Hultgren}). By the continuity of the functions, and matching the functions on the intersection of different faces, we see that the constant must be independent of the face.
\end{proof}

\subsection{Variational inequality}

We now derive a fundamental \emph{variational inequality for the minimizer}.

\begin{prop}\label{variationalinequality}
For any measurable $E\subset \partial \Delta_\lambda^\vee$, we have
$
|\bar{\nabla}\phi(E)  |\geq |E|.
$
Completely analogously $
|\bar{\nabla}\phi^*(E)  |\geq |E|
$
for any measurable subset $E\subset \partial\Delta_\mu$. 
\end{prop}

\begin{proof}
We consider the function $\phi-t1_E$ for $0<t\ll 1$, where $1_E$ denotes the characteristic function of $E$. By Lemma \ref{MonotonicityLegendre},
\[
\phi^*\leq (\phi-t1_E)^*\leq \phi^*+t.
\]
Morever $(\phi-t1_E)^*>\phi^*$ at $p\in \partial\Delta_\mu$, only when 
\[
\sup_{x\in E}( \langle p, x\rangle -\phi(x) )\geq \phi^*(p)-t.
\]
Hence
\[
\begin{split}
t |\{ p: \sup_{x\in E} ( \langle p, x\rangle -\phi(x)  ) \geq \phi^*(p)-t   \} |
\geq  \int_{\partial\Delta_\mu} ((\phi-t1_E)^*-\phi^*)  d\mathcal{L}.
\end{split}
\]

Since $\phi$ is a minimizer, plugging in $(\phi-t1_E)^{**}$ as a competitor (the caveat being that $(\phi-t1_E)$ may not lie in $\mathcal{P}$), we get
\[
\int_{\partial \Delta_\mu} (\phi-t1_E)^*-\phi^* \geq - \int_{\partial \Delta_\lambda^\vee} (\phi-t1_E)^{**}- \phi \geq - \int_{\partial \Delta_\lambda^\vee} (\phi-t1_E)- \phi=t|E|.
\]
Here the second inequality uses Lemma \ref{InvolutionLegendre}. Combining the above and cancelling the $t$ factor, 
\[
|\{ p: \sup_{x\in E} ( \langle p, x\rangle -\phi(x)  ) \geq \phi^*(p)-t   \} | \geq |E|,\quad \forall 0<t\ll 1. 
\]
Taking the $t\to 0$ limit, and recalling $\phi^*(p)\geq \langle x,p\rangle-\phi(x)$, we get
\[
|\{ p: \sup_{x\in E} ( \langle p, x\rangle -\phi(x)  ) = \phi^*(p)   \} | \geq |E|.
\]

For closed subsets $E$, the sup can be replaced by max by the continuity of $\phi$, so the LHS subset is the conjugate set  $\bar{\nabla}\phi(E)$. For general Borel subsets $E$, we take a compact exhaustion of $E$. For any compact $K\subset E$, we observe
\[
|\bar{\nabla}\phi(E) |\geq |\bar{\nabla} \phi(K)|\geq |K|.
\]
Since the Lebesgue measure is inner regular, we obtain $|\bar{\nabla}\phi(E)| \geq |E| $ by taking the limit $K\uparrow E$.
\end{proof}

\begin{cor}\label{variationalinequalitycor}
If no point in the measurable subset $E\subset \partial \Delta_\lambda^\vee$ has any anomalous conjugate point, then $|\nabla \phi(E)|\geq |E|$.
\end{cor}

\begin{proof}
If there is no anomalous conjugate point, then $\nabla \phi(E)= \bar{\nabla}\phi(E)$. We then apply the variational inequality.
\end{proof}

\begin{rmk}
We can now explain the heuristic why the variational problem should be related to the \emph{real MA equation}. In the ideal situation, there is no anomalous conjugate point, and $\nabla \phi, \nabla \phi^*$ are both bijective on points, and define mutually inverse maps. Then \[
|\nabla \phi(E)|\geq |E|= |\nabla\phi^*(\nabla \phi (E))|\geq |\nabla \phi(E)|.
\]
This forces equality everywhere. On each open face of $\partial \Delta_\lambda^\vee$ the gradient agrees with the classical notion in convex function theory, and $|\nabla \phi(E)|=|E|$ is just the weak formulation of the real MA equation. To make this argument work more rigorously, one needs to control the anomalous conjugate points, and the multivalued nature of gradient.
\end{rmk}

\subsection{Anomalous conjugate points}

The remaining technical difficulty mainly comes from the possible existence of anomalous conjugate points.

We introduce the subset $\mathcal{S}^\vee\subset \partial \Delta_\lambda^\vee$ as $\mathcal{S}^\vee=\mathcal{S}^\vee_1\cup \mathcal{S}^\vee_2$, where
\[
\mathcal{S}^\vee_1= \bigcup\{ \text{lower dim faces of $\partial \Delta_\lambda^\vee$}  \},
\]
\[
\mathcal{S}^\vee_2=\bigcup_{m} \{ x| \bar{\nabla} \phi(x)\cap H_m^+ \text{ contains at least two points}        \}.
\]
In particular $\nabla\phi(x)$ is single valued outside $\mathcal{S}^\vee$. Similarly the reader may write down the Legendre dual analogue $\mathcal{S}$. We now recall a basic fact in classical convex function theory.

\begin{lem}\label{nomultiplegradient}
The subset $\mathcal{S}^\vee\subset \partial \Delta_\lambda^\vee$ 
 has Lebesgue measure zero. A similar statement holds for $\mathcal{S}\subset \partial \Delta_\mu$.
\end{lem}

\begin{proof}
(\cf \cite[section 1.1.1]{Gutierez})
We restrict $\phi$ to the the open face $\text{Int}(\Delta_m^\vee)$, to obtain a classical convex function on a convex domain. As explained in section \ref{conjugategradient}, the notion of gradient inside $H_m^+$ is naturally identified with the classical notion which takes value in $M_\R/\R m$. Since our convex functions are Lipschitz, they are differentiable Lebesgue-a.e.
The multi-valued gradient problem happens on the non-differentiability locus, which has measure zero. Now the global statement follows by taking the union of all the faces.
\end{proof}

The following lemmas say that the anomalous conjugate has controlled effect on the gradient set. In this section, let $E$ be a measurable subset of $\Delta_m^\vee$, such that any $x\in E$ admits some anomalous conjugate point in $H_m^-$. Let $F'= \bar{\nabla}\phi(E)\cap H_m^-$, and let $F\subset H_m^+$ be the image of $F'$ under the projection $H_m^-\to H_m^+$ (\cf section \ref{conjugategradient}). By the proof of Lemma \ref{gradientexistencelem}, we have $F\subset \bar{\nabla}\phi(E)\cap H_m^+$.

\begin{lem}\label{lemmaF}
%Let $E$ be a measurable subset of $\Delta_m^\vee$, such that any $x\in E$ admits some anomalous conjugate point $p'$ with $\langle m, p'\rangle+ \lambda(m)<0$. We write $F= \bar{\nabla} \phi(E)\cap H_m^+$, then
 $|F|\leq |E|$. 
%In particular, since $F$ contains $\nabla\phi(E)$, we have $|\nabla\phi(E)|\leq |E|$. 

\end{lem}

\begin{proof}
%For any $x\in \partial\Delta_\lambda^\vee$, 
%the subset of the convex polytope $\Delta_\mu$ 
%\[
%\{  p:   \langle p, x\rangle- \phi^*(p) =\phi(x)  \}
%\]
%must be convex, since it is the maximum set of a concave function.
By construction every $p\in F\setminus \mathcal{S}$ is of the form $p=p'+sm$ with $s>0$ and $p'\in F'$. By Cor. \ref{anomalousrigidity} concerning the anomalous conjugate points, we must have $\bar{\nabla}\phi^*(p)\subset \Delta_m^\vee$. Since $p\in H_m^+\setminus \mathcal{S}$, it lies on only one face $\Delta_n\subset H_m^+$ with $\langle n,m\rangle=1$, whence $\Delta_m^\vee\subset H_n^+$. We then have
$
\bar{\nabla}\phi^*(p)\subset \Delta_m^\vee\subset H_n^+,
$
whence
\[
\bar{\nabla}\phi^*(p)= \bar{\nabla}\phi^*(p)\cap H_n^+ = \nabla \phi^*(p).
\]
Since $p\notin \mathcal{S}$, we know that $\nabla \phi^*(p)$ consists of only one point, so in fact $\nabla \phi^*(p)=\{x\}$. But $p\in F\subset  \bar{\nabla}\phi(E)$, so $x\in E$ and  $\bar{\nabla}\phi^*(p)\subset E$.

In summary we get the inclusion
$
\bar{\nabla}\phi^*(F\setminus \mathcal{S}) \subset E.
$
Now applying the variational inequality Prop. \ref{variationalinequality}, 
\[
|E|\geq |\bar{\nabla}\phi^*(F\setminus \mathcal{S})|\geq |F\setminus \mathcal{S}|= |F|.
\]
The last equality uses that $\mathcal{S}$ has zero measure, by Lemma \ref{nomultiplegradient}.
\end{proof}

\begin{lem}\label{lemmaF'}
 $|F'|\leq |F|$. When the equality is achieved, then $F'$ has zero measure except perhaps on the faces $\Delta_n$ with $\langle n,m\rangle=-1$.

\end{lem}

\begin{proof}
By Lemma \ref{Hm+-}, under the quotient map $M_\R\to M_\R/\R m$, the sets $H_m^+$ and $H_m^-$ each project homeomorphically to the image of $\Delta_\mu$. %By the proof of Lemma \ref{gradientexistencelem}, the projection sends the set $F'$ onto a subset of $F\subset H_m^+$.  %To prove $|F'|\leq |F|$, we need to understand how Lebesgue measures behave under the projection map.
The projection map is piecewise affine, and we focus on some face $\Delta_n$. 
With respect to the natural Lebesgue measures $dp$ induced by the integral structure, the quotient map $\Delta_{n}\to M_\R/\R m$ has Jacobian factor $|\langle n, m\rangle|$. For $\Delta_n\subset H_m^+$ this factor is equal to one, and for $\Delta_n\subset H_m^-$ this factor is a positive integer (\cf Lemma \ref{Hm+-}). Thus under the projection $H_m^-\to H_m^+$, the Jacobian factor is $\geq 1$. Since the normalized measure $d\mathcal{L}$ is a constant multiple of $dp$, the Jacobian factor is still $\geq 1$ with respect to $d\mathcal{L}$. This shows $|F'|\leq |F|$. The equality case forces the Jacobian factor to be equal to one, which means $|\langle n,m\rangle|=1$.
\end{proof}

\begin{prop}\label{anomaloustypeI1}
$|E|=|F|=|F'|$. In particular the equality forces that $F'$ has zero measure except perhaps on the faces $\Delta_n$ with $\langle n,m\rangle=-1$.

\end{prop}

\begin{proof}
By Lemma \ref{lemmaF}, \ref{lemmaF'},
we already know $|E|\geq |F|\geq |F'|$, so it suffices to prove $|F'|\geq |E|$. We write  $H_m^-=\bigcup \Delta_{n_i}$. For each $i$, we have the set inclusion $\Delta_m^\vee\subset H_{n_i}^-$, and since any point of $E$ is conjugate to some point in $H_m^-\cap \bar{\nabla}\phi(E)=F'$, we conclude
 \[
E\subset   \bar{\nabla}\phi^*(F' )\cap \Delta_m^\vee \subset \bigcup_i \bar{\nabla}\phi^*(F'\cap \Delta_{n_i} )\cap H_{n_i}^-,
 \]
whence
$
|E|\leq \sum_i | \bar{\nabla}\phi^*(F'\cap \Delta_{n_i} )\cap H_{n_i}^-|. 
$

Recall that our setup has the Legendre duality symmetry. For any fixed $i$, we now let $F'\cap \Delta_{n_i}$ play the role of $E$ in the Legendre dual version of Lemma \ref{lemmaF}, \ref{lemmaF'}. The role of $F'$ is then played by $\bar{\nabla} \phi^*(F'\cap \Delta_{n_i} )\cap H_{n_i}^-$, and we conclude
\[
| \bar{\nabla} \phi^*(F'\cap \Delta_{n_i} )\cap H_{n_i}^- |\leq |F'\cap \Delta_{n_i}|.
\] 
Summing over $i$ and
combining the above,
\[
|E|\leq \sum_i | \bar{\nabla}\phi^*(F'\cap \Delta_{n_i} )\cap H_{n_i}^-|\leq \sum_i |F'\cap \Delta_{n_i}|= |F'|.
\]
This is the promised reverse inequality.
\end{proof}

\begin{rmk}
The equality forces a very strong rigidity, and it is an interesting question whether it forces $|E|=0$ in fact.
\end{rmk}

\begin{cor}\label{anomoloustypeI2}
 $|E|=|\nabla\phi(E)\cap F|\leq |\nabla\phi(E)|$.	
\end{cor}

\begin{proof}
Since  $|E|=|F|$ by Prop. \ref{anomaloustypeI1}, it suffices to prove $F\setminus \nabla\phi(E)\cup\mathcal{S} $ has measure zero.
Now every point in $p\in F\setminus \mathcal{S}$ has a unique gradient $x=\nabla\phi^*(p)\in E$, and if $x\notin \mathcal{S}^\vee$ then $p=\nabla \phi(x)$. This shows that 
\[
F\setminus \nabla\phi(E)\cup\mathcal{S} \subset \text{Image of }(\bar{\nabla}\phi (\mathcal{S}^\vee\cap E)\cap H_m^-) \text{ under } H_m^-\to H_m^+ .
\]
By applying Prop. \ref{anomaloustypeI1}
again to $E\cap \mathcal{S}^\vee$, we see
\[
|F\setminus \nabla\phi(E)\cup\mathcal{S}^\vee|\leq |\text{Image}(\bar{\nabla}\phi (\mathcal{S}^\vee\cap E)\cap H_m^- ) |
= |E\cap \mathcal{S}^\vee|=0,
\]
so the result follows.
\end{proof}

\subsection{Anomalous conjugate points II}

In this section we aim to prove 

\begin{prop}\label{gradientdecreasesize}
For any $E\subset \partial\Delta_{\lambda}^\vee$, we have $|\nabla \phi (E)|\leq |E|$.
\end{prop}

\begin{lem}\label{anomalousconjugatetypeII1}
Assume the subset $E\subset \Delta_m^\vee$ has the property that any $x\in E$ admits some conjugate in a face $\Delta_n$ with $\langle m,n\rangle=0$. Then $|E\cap \nabla\phi^*(\partial \Delta_\mu)|=0$.

\end{lem}

\begin{proof}
Consider a point $x\in (E\cap \nabla\phi^*(\partial \Delta_\mu)) \setminus  \mathcal{S}^\vee$. Since $x\in \nabla\phi^*(\partial \Delta_\mu)$, we write $x=\nabla\phi^*(p)$. Since $x$ does not lie on $\mathcal{S}^\vee$, it has a unique gradient, so $p= \nabla\phi(x)$.

Now by assumption $x$ has some anomalous conjugate $p'$ in a face $\Delta_n$ with $\langle m,n\rangle =0$. By Lemma \ref{existence}, there exists a gradient of the form $p'+sm$ for $s>0$. The uniqueness of the gradient then forces $p=p'+sm$. Since $p'\in \Delta_n$, we have $\langle p', n\rangle+\mu(n)=0$, hence
\[
\langle p, n\rangle +\mu(n)= \langle p-p', n\rangle = s\langle m, n\rangle= 0.
\]
This means $p\in \Delta_n$. Since $x$ lies on $\Delta_m^\vee$ but not on $\mathcal{S}^\vee$, it must be in the interior of $\Delta_m^\vee$, so by the definition of gradient $x=\nabla\phi^*(p)$ forces $\langle m, n \rangle=1$. This contradicts $\langle m, n\rangle=0$. This contradiction shows that $x$ cannot exist, which means $(E\cap \nabla\phi^*(\partial \Delta_\mu^\vee)) \setminus  \mathcal{S}^\vee$ is empty. Since $\mathcal{S}^\vee$ has measure zero by Lemma \ref{nomultiplegradient}, we see 
$|E\cap \nabla\phi^*(\partial \Delta_\mu)|=0$.
\end{proof}

We also record the Legendre dual statement.

\begin{lem}\label{anomalousconjugatetypeII3}
Suppose the subset $G\subset \Delta_n$ has the property that any $p\in G$ admits a conjugate point in $\Delta_m^\vee$ with $\langle m,n\rangle=0$. Then $|G\cap \nabla\phi(\partial \Delta_\lambda^\vee)|=0$.
\end{lem}

We can now prove Prop. \ref{gradientdecreasesize}.

\begin{proof}
(Prop. \ref{gradientdecreasesize})
First, we notice it suffices to prove Prop. \ref{gradientdecreasesize} only for $E\subset \Delta_m^\vee$. This is because we can decompose $E$ into the union of $E_m\subset \Delta_m^\vee$ for all the possible $m$, and once we know $|\nabla \phi(E_m)|\leq |E_m|$ for every $m$, it would follow that
\[
|\nabla \phi(E)|\leq \sum|\nabla \phi(E_m)|\leq \sum |E_m|=|E|.
\]
The same argument shows more generally that if we can partition $E$ into a union of subsets, and the intersections have measure zero, then it suffices to prove the statement for the subsets.

We will partition $\nabla\phi(E)\setminus \mathcal{S}$ into several subsets. 

\begin{enumerate}
\item 
 Let $G_1$ be the union over $n$ of all the $p\in\nabla\phi(E)\cap \Delta_n\setminus \mathcal{S}$, which admit an anomalous conjugate point in some $\Delta_{m'}^\vee$ with $\langle m',n\rangle=0$. By Lemma \ref{anomalousconjugatetypeII3}, its size $|G_1|=0$.

 \item 
 Let $G_2$ consist of all the $p\in \nabla\phi(E)\setminus \mathcal{S}$ which admit an anomalous conjugate point, but not covered by the first case. We decompose $G_2\subset \nabla\phi(E)\subset H_m^+$ into the union of $G_{2,n}=G_2\cap \Delta_n$ for all $\Delta_n\subset H_m^+$. By the Legendre dual version of Cor. \ref{anomoloustypeI2} applied to $G_{2,n}$, we have
$
 |G_{2,n}|\leq |\nabla\phi^*(G_{2,n})|, 
$
 whence
 \[
 |G_2|= \sum |G_{2,n}|\leq \sum |\nabla\phi^*(G_{2,n})|. 
 \]

 Next,
for any $p\in \nabla\phi(E)\setminus \mathcal{S}$, 
  we can write $p=\nabla\phi(x)$ for $x\in E$, and since $p$ is not in $\mathcal{S}$, it has a unique gradient point which is $x\in E$. Thus $\nabla\phi(\nabla\phi^*(p))$ contains $p$, and $\nabla\phi^*(p)\subset E\subset \Delta_m^\vee$. Since the subset of points in $\Delta_m^\vee$ with multiple gradients has zero measure (\cf Lemma \ref{nomultiplegradient}), we see that the intersections between the sets $\nabla\phi^*(G_{2,n})$ have zero measure, whence
  \[
  |G_2|\leq\sum |\nabla\phi^*(G_{2,n})|= |\nabla\phi^*(G_2)|. 
  \]

\item Let $G_3$ consist of all the $p\in \nabla\phi(E)\setminus \mathcal{S}$ with no anomalous conjugate point. Any such $p$ lies outside of $\mathcal{S}$, hence has a unique gradient, which must then be its unique conjugate point. The set $\nabla\phi^*(G_3)\subset E$ is the union of these conjugate points. By the Legendre dual version of Cor. \ref{variationalinequalitycor} we have $|\nabla\phi^*(G_3)|\geq |G_3|$.
 
\end{enumerate}

By construction $\nabla\phi(E)\setminus \mathcal{S}=G_1\cup G_2\cup G_3$, hence by the above discussions
\[
\begin{split}
|\nabla\phi(E)\setminus \mathcal{S}|\leq |G_1|+|G_2|+|G_3|\leq |\nabla\phi^*(G_2)|+|\nabla \phi^* (G_3)|.
\end{split}
\]
The same argument in item 2 above shows that 
$|\nabla\phi^*(G_2)\cap \nabla\phi^*(G_3)|=0$, so
\[
|\nabla\phi(E)\setminus \mathcal{S}|\leq
|\nabla\phi^*(G_2)|+|\nabla \phi^* (G_3)|
=|\nabla\phi^*(G_2\cup G_3)|\leq |E| . 
\]
The last inequality here is because $\nabla\phi^*(G_2\cup G_3)\subset E$. 
Now since $\mathcal{S}$ has measure zero, we get 
\[
|\nabla \phi(E)|=|\nabla\phi(E)\setminus \mathcal{S}|\leq |E|,
\]
as required.
\end{proof}

\subsection{Structure theorems for the minimizer}

In this section we will list certain structural properties of the minimizer $\phi$, which in the best situation lead to the real MA equation.

\begin{prop}\label{Borelmeasure}
The set function $E\mapsto |\nabla\phi(E)|$ defines a \emph{Borel measure} on $\partial \Delta_\lambda^\vee$, which is absolutely continuous with respect to the Lebesgue measure. Likewise with the Legendre dual analogue.
\end{prop}

\begin{proof}
Since $|\nabla \phi(E)|\leq |E|$ by Prop. \ref{gradientdecreasesize}, we see that $|E|=0$ implies $|\nabla \phi(E)|=0$. Next we justify the countable additivity axiom. Let $E=\bigcup E_i$ be a disjoint partition, then clearly $|\nabla \phi(E)|\leq \sum |\nabla\phi(E_i)|$. To see that the equality holds, without loss we may assume $E$ is disjoint from the measure zero set $\mathcal{S}^\vee$, so that the gradient is a single valued map.  Morever, the mutual intersections of $\nabla\phi(E_i)$ is contained in $\mathcal{S}$, which again has measure zero, so
\[
\sum |\nabla \phi(E_i)|= |\bigcup_i\nabla \phi(E)|= |\nabla\phi(E)|
\]
as required.
\end{proof}

We now introduce a \emph{good-bad decomposition} on $\partial \Delta_\lambda^\vee$. The \emph{good set} $\mathcal{G}^\vee\subset \partial \Delta_\lambda^\vee$ consists of the following two types of points $x\in \partial \Delta_\lambda^\vee$:
\begin{itemize}
\item Type I good point: $x$ has no anomalous conjugate point, so that all conjugate points are gradients of $x$;
\item Type II good point:  $x$ is contained in some face $\Delta_m^\vee$ and $x$ admits an anomalous conjugate point $p\in \bar{\nabla}\phi(x)\cap H_m^-$.
\end{itemize}
The \emph{bad set} $\mathcal{B}^\vee\subset \Delta_\lambda^\vee$ consists of all $x$ contained in some face $\Delta_m^\vee$, and which admits some anomalous conjugate point $p\in \Delta_n$ with $\langle m, n\rangle=0$.  Clearly $\partial \Delta_\lambda^\vee=\mathcal{G}^\vee\cup \mathcal{B}^\vee$.
We leave the reader to write down the Legendre dual version $\partial \Delta_\mu=\mathcal{G}\cup \mathcal{B}$.

\begin{thm}\label{goodbaddecomposition}
(Property of \emph{good-bad decomposition})
If $E\subset \mathcal{G}^\vee$, then $|\nabla \phi(E)|=|E|$. If $E\subset \mathcal{B}^\vee$, then $|\nabla\phi(E)|=0$. In particular $|\mathcal{G}^\vee\cap \mathcal{B}^\vee|=0$.
\end{thm}

\begin{proof}
Since the measure $E\mapsto |\nabla\phi(E)|$ has finite additivity, it suffices to consider $E$ of single types. For type I good points, we use $|\nabla\phi(E)|\leq |E|$ from Prop. \ref{gradientdecreasesize} and $|\nabla\phi(E)|\geq |E|$ from Cor. \ref{variationalinequalitycor}. For type II good points, we use Cor. \ref{anomoloustypeI2}, Prop. \ref{gradientdecreasesize}  and finite additivity.

Now we consider the bad points. By finite additivity, without loss $E\subset \Delta_m^\vee$. Since $|\mathcal{S}^\vee|=|\nabla\phi(\mathcal{S}^\vee)|=0$, without loss we can suppose $E$ is disjoint from $\mathcal{S}^\vee$, so every $x\in E$ has a unique gradient $p=\nabla\phi(x)$. By assumption $E$ consists of bad points, so there is an anomalous conjugate $p'\in \Delta_n$ with $\langle m, n\rangle=0$. By Lemma \ref{gradientexistencelem}, the unique gradient is of the form $p=p'+sm$ for $s>0$. Thus
\[
\langle p, n\rangle=\langle p',n\rangle +s\langle n, m\rangle= - \mu (n),
\]
whence $p$ lies on $\Delta_n$ as well. Since $p$ is a gradient of $x$, it must also lie on $H_m^+$, so it falls into $H_m^+\cap \Delta_n$, which has measure zero.
\end{proof}

One lesson of Prop. \ref{goodbaddecomposition} is that $\nabla\phi$ only sees the good set, and an analogous version holds for $\nabla\phi^*$. In fact modulo measure zero sets, these two maps are in some sense inverse to each other, when we restrict to the good set.

\begin{thm}\label{Legendreduality}
(\emph{Legendre duality}) The maps $\nabla\phi$ and $\nabla\phi^*$ have the following properties:
\begin{enumerate}
\item For any $E\subset \mathcal{G}^\vee$, we have $|\nabla\phi(E)\cap \mathcal{B}|=0$, and $|\nabla\phi(E)\cap \mathcal{G}|=|E|$.

\item 
We consider the maps between subsets of $\mathcal{G}^\vee$ and $\mathcal{G}$ defined by
\[
\begin{cases}
E\mapsto \nabla\phi(E)\cap \mathcal{G}, \quad \forall  E\subset \mathcal{G}^\vee,
\\
F\mapsto \nabla\phi^*(F)\cap \mathcal{G}^\vee, \quad \forall F\subset \mathcal{G}.
\end{cases}
\]
The mutual compositions agree with $E, F$ up to a measure zero subset.

\item In particular $|\mathcal{G}|=|\mathcal{G}^\vee|$ and $|\mathcal{B}|=|\mathcal{B}^\vee|$.
\end{enumerate}
\end{thm}

\begin{proof}
The claim that $|\nabla\phi(E)\cap \mathcal{B}|=0$ follows from the Legendre dual version of Lemma \ref{anomalousconjugatetypeII1}. By Prop. \ref{Borelmeasure} and Thm. \ref{goodbaddecomposition}, we have $|\mathcal{G}\cap \mathcal{B}|=0$, and
\[
|\nabla \phi(E)\cap \mathcal{G}|= |\nabla \phi(E)|- |\nabla\phi(E)\cap \mathcal{B}|=|\nabla \phi(E)|=|E|.
\]
This proves item 1.

Consequently, the two maps in item 2 \emph{preserve the measures}. To prove the composition statement, without loss $E$ is disjoint from the measure zero set $\mathcal{S}^\vee$. For any $p\in \nabla\phi(E)\cap \mathcal{G}\setminus \mathcal{S}$, the $\nabla\phi^*(p)$ uniquely recovers the point $x\in E$ with $p=\nabla\phi(x)$. Thus the composition of $\nabla\phi$ and $\nabla\phi^*$ recovers a full measure subset of $E$. This proves item 2.

Now by item 1, 2 applied to $\mathcal{G}^\vee$, we get $|\mathcal{G}|=|\mathcal{G}^\vee|$. Since $\partial \Delta_\lambda^\vee= \mathcal{G}^\vee\cup \mathcal{B}^\vee$ and $|\mathcal{G}^\vee\cap \mathcal{B}^\vee|=0$, we have
\[
|\mathcal{B}|=1-|\mathcal{G}|= 1-|\mathcal{G}^\vee|=|\mathcal{B}^\vee|,
\]
which proves item 3.
\end{proof}

The most important special case for us is the following:

\begin{thm}\label{realMAthm}
The following conditions are equivalent:
\begin{enumerate}
\item The size of the bad set $|\mathcal{B}|=0$, or equivalently $|\mathcal{B}^\vee|=0$. (For instance, this works if $\langle n, m\rangle=0$ never holds for any vertices of $\Delta, \Delta^\vee$.)

\item  The images of $\nabla\phi$ and $\nabla\phi^*$ have full measure. 
\end{enumerate}
When this holds, then $|\nabla\phi(E)|=|E|$ for any $E\subset \partial \Delta_\lambda^\vee$, and likewise with $\nabla\phi^*$. The two maps $\nabla\phi$ and $\nabla\phi^*$ compose to the identity modulo measure zero sets.

\end{thm}

\begin{proof}
By Thm. \ref{Legendreduality}, the vanishing of $|\mathcal{B}|$ would force $|\mathcal{B}^\vee|=0$. Thus the good sets have full measure, and we know from Thm. \ref{Legendreduality} that $\nabla\phi$ and $\nabla\phi^*$ on the good sets are measure preserving, mutually inverse maps modulo zero measure. In particular the images of $\nabla\phi$ and $\nabla\phi^*$ have full measure.

Conversely, suppose the images of $\nabla\phi$ and $\nabla\phi^*$ have full measure. By Thm. \ref{goodbaddecomposition}, 
\[
|\nabla \phi (\partial \Delta_\lambda^\vee)|= |\nabla \phi (\mathcal{G}^\vee)|+ |\nabla \phi (\mathcal{B}^\vee)| =|\mathcal{G}^\vee|,
\]	
which then forces $|\mathcal{G}^\vee|=1$, namely $|\mathcal{B}^\vee|=0$. This shows that item 2 implies item 1, and the rest follow from Thm. \ref{Legendreduality}.
\end{proof}

When the equivalence conditions in Thm. \ref{realMAthm} is satisfied, then the conclusions admit classical interpretations. When we restrict $\phi$ to an open face $\text{Int}(\Delta_m^\vee)$ to obtain a convex function still denoted as $\phi$, then $\nabla\phi$ is identified with the classical gradient (\cf section \ref{conjugategradient}). Recall from (\ref{normalizedmeasure}) how the normalized measure is related to the Lebesgue measure induced by the integral structure. Morever, the projection map  $H_m^+\to M_\R/\R m$ has Jacobian factor one with respect to the integral Lebesgue measure (\cf the proof of Lemma \ref{lemmaF'}). In summary, the fact that $|\nabla\phi(E)|=|E|$ with respect to the normalized measures, translates into the weak Alexandrov formulation of the following real MA equation over $\text{Int}(\Delta_m^\vee)$:
\begin{equation}\label{realMA1}
\det(D^2\phi)= C_0= \frac{  \int_{\partial \Delta_{\mu}} dp   }{ \int_{\partial \Delta_{\lambda}^\vee} dx  }.
\end{equation}
Completely analogously, we have the weak Alexandrov formulation of the following real MA equation over the open faces $\text{Int}(\Delta_n)$:
\begin{equation}
\det(D^2\phi^*)= C_0^{-1}.
\end{equation}
This expresses the local PDE nature of the variational problem.

%The minimizer $\phi\in \mathcal{P}$ extends canonically to a convex function $\phi$ on $N_\R$ (\cf section \ref{Legendretransformonpolytopes}). We observe a simple invariance property:

%\begin{cor}
%For $x\in \Delta_m^\vee$, suppose its admits a gradient $p\in \Delta_n\subset H_m^+$. Then the value of the convex function $\phi+\lambda(n)\langle \cdot,m\rangle$ is constant on the ray $x+\R_{\geq 0} n$ inside $N_\R$.

%\end{cor}

%\begin{proof}
%Any $x\in \Delta_m^\vee$ admits a gradient $p\in \Delta_n\subset H_m^+$, so for $y=x+sn$, we have
%\[
%\phi(y)\geq \langle y,p\rangle- \phi^*(p)= \langle x, p\rangle-\phi^*(p) + s\langle n, p\rangle =\phi(x) -s\lambda(n).
%\]
%On the other hand, since the convex function $\phi$ on $N_\R$ has gradient contained in $\Delta_\mu$, we have
%\[
%\phi(y)\leq \phi(x)+ \max_{\Delta_\mu} \langle y-x, p\rangle = \phi(x) + s\max_{\Delta_\mu} \langle n, p\rangle = \phi(x)-s\lambda(n).
%\]
%Hence equality is achieved, which proves the claim.
%\end{proof}

\subsection{Some examples}

We first give some simple examples to which Theorem \ref{realMAthm} applies.

\begin{eg}
We revisit the standard simplex example \cite[section 1.1]{Hultgren}\cite{LiFermat}. 
We take the standard basis $e_0,e_1,\ldots e_{d+1}$ inside $\Z^{d+2}$, and define
\[
M= \{ p\in \Z^{d+2}|\sum p_i=0   \},\quad N= M^\vee= \Z^{d+2}/\Z(1,\ldots 1),\quad 
\]
and $M_\R=M\otimes \R,  N_\R=N\otimes \R.$ The polytope $\Delta\subset M_\R$ is the convex hull of \[
m_i=(d+2)e_i- \sum_0^{d+1} e_j,\quad i=0,\ldots d+1.
\]
Its dual polytope $\Delta^\vee$ is also a simplex, with vertices given by
\[
n_0=(-1,0,\ldots, 0),\ldots  n_{d+1}=(0,\ldots, 0,-1).
\]
Thus \[
\langle m_i, n_j\rangle=\begin{cases}
-(d+1),\quad & i=j,
\\
1,\quad & i\neq j.
\end{cases}
\]
In particular $\langle n,m\rangle =0$ never occurs, and Thm \ref{realMAthm} applies. This example corresponds to hypersurfaces inside $\mathbb{CP}^{d+1}$. 
\end{eg}

\begin{eg}
Suppose we have $k$ integral reflexive polytopes $\Delta_i\subset M_{\R,i}$, with dual polytopes $\Delta_i^\vee\subset N_{\R,i}$, such that 
the pairing $\langle m,n\rangle \neq 0$ for all possible vertices $m, n$ of $\Delta_i$ and $\Delta^\vee_i$. Then we can take
\[
N_\R=\oplus_i N_{\R,i},\quad M_\R= \oplus_i M_{\R,i},\quad \Delta\subset M_\R,\quad \Delta^\vee\subset N_\R.
\]
The vertices of $\Delta^\vee$ are of the form $(0,\ldots ,n_i, 0,\ldots)$ with $n_i\in \text{vertex}(\Delta_i^\vee)$, while the vertices of $\Delta$ are of the form $(m_1,\ldots m_k)$ for $m_i\in \text{vertex}(\Delta_i)$.  It is immediate to check that the pairing $\langle m,n\rangle\neq 0$ still holds for $\Delta,\Delta^\vee$. Algebro-geometrically, this is just taking the product of several toric Fano varieties. 
\end{eg}

We now give a simple example which shows that $|\mathcal{B}|=0$ might fail if $\langle m,n\rangle=0$ is allowed.

\begin{eg}
This example lives inside $\R^2$. Let $\Delta$ be the convex full of 
\[
(1,0), (0,1), (-1,1), (-1,0), (0,-1), (1,-1).
\]
Then $\Delta^\vee$ is the convex hull of 
$
\pm (1,0), \pm (1,1), \pm (0,1).
$
Notice for instance that $m_0=(1,-1)$ is orthogonal to $n_0=(1,1)$. We can take $\Delta_\lambda^\vee=\Delta^\vee$, and let
\[
\Delta_\mu= \{ (x,y)\in \R^2: |x|\leq 1, |y| \leq 1,  |x+y|\leq \epsilon    \},
\]
for some given constant $0<\epsilon<1$. When $\epsilon\ll 1$, then $\Delta_\mu$ is concentrated along a line segment.

We claim that $|\mathcal{B}|\neq 0$. Otherwise by Thm \ref{realMAthm}, we have $|\nabla\phi(E)|=|E|$ for any $E$. We take $E=\Delta_{m}^\vee$, so that by the definition of gradient $\nabla\phi(E)\subset H_{m}^+$. Thus
\[
|\Delta_{m}^\vee| = |\nabla\phi(\Delta_{m}^\vee)|\leq |H_{m}^+|.
\]
For the choice $m=m_0=(1,-1)$, this leads to a contradition for small enough $\epsilon$.

Algebro-geometrically, the corresponding Fano manifold $X_\Delta$ is the blow up of $\mathbb{P}^1\times \mathbb{P}^1$ at two toric fixed points, so there are two disjoint exceptional divisors $E_1, E_2$. Adjusting $\epsilon$ amounts to changing the K\"ahler class on $X_\Delta$ in a 1-parameter family, with $\epsilon=1$ corresponding to $c_1(\pi_1^*[\mathcal{O}_{\mathbb{P}^1}]+ \pi_2^*[\mathcal{O}_{\mathbb{P}^1}])$, and decreasing $\epsilon$ corresponding to subtracting multiplies of $E_1+E_2$. The family $X_t$ of Calabi-Yau hypersurfaces (\cf the introduction) is here just a pencil of elliptic curves. The second cohomology of an elliptic curve has rank one, so the restriction of all these K\"ahler classes to the elliptic curve will all be proportional. A plausible interpretation of the above example, is that once an ample polarization $L\to X$ on the degenerating Calabi-Yau hypersurfaces is fixed, there is still some flexibility to extend this polarization to $L\to X_\Delta$, and for an inappropriate choice of $L\to X_{\Delta}$, the variational problem may not be helpful for the SYZ conjecture.

\end{eg}

\subsection{Open questions}

Our work suggests a number of natural questions, which are relevant for further applications to the SYZ conjecture, and which may be of some independent interest to PDE theorists.

\begin{enumerate}
\item 
Thm. \ref{realMAthm} makes it clear that the failure for the minimizer to satisfy $|\nabla\phi(E)|=|E|$ is measured by the \emph{size of the bad set} $|\mathcal{B}|$. Is it possible to compute or estimate $|\mathcal{B}|$ in general, when we allow $\langle m,n\rangle=0$?

\item
Assume $|\mathcal{B}|=0$. Is it possible for the anomalous conjugate point to exist?

\item 
Assume $|\mathcal{B}|=0$. When can we prove that \emph{the gradient is unique everywhere} (instead of Lebesgue-a.e), namely $\nabla\phi(x)$ consists of only one point for every $x\in \partial \Delta_\lambda^\vee$?

\item 
The Legendre dual version of the above question asks if  $\nabla\phi^*(p)$ consists of only one point for every $p\in \partial \Delta_\mu$.

This question is closely related to the \emph{strict convexity} of the restriction $\phi$ to a convex function on the open faces.
This has implication on the regularity question of $\phi$, since a strictly convex solution of the real MA equation is known to be smooth.

If both $\nabla\phi$ and $\nabla\phi^*$ are single valued, they would define mutually inverse maps setting up a homeomorphism between $\partial \Delta_\lambda^\vee$ and $\partial \Delta_\mu$. This global version of Legendre duality is desirable from the viewpoint of potential applications to mirror symmetry.

\item
The polytopes $\partial\Delta_\lambda^\vee$ and $\partial \Delta_\mu$ have no natural global smooth structure; the only a priori information is that the open faces $\text{Int}(\Delta_m^\vee)$ and $\text{Int}(\Delta_n)$ have natural affine structures, and in particular smooth structures.

Now suppose the previous questions have positive answers, so that $\phi$ and $\phi^*$ are smooth on the respective open faces, and the gradient maps set up a global homeomorphism. Then we can transfer $\text{Int}(\Delta_n)$ to the open subset $\nabla\phi^*(\text{Int}(\Delta_n) )\subset \partial \Delta_\lambda^\vee$, and on the overlap with the open faces of $\partial \Delta_\lambda^\vee$ the smooth structures are compatible. This suggests that there is a \emph{hidden smooth structure with singularity} on $\partial \Delta_\lambda^\vee$, where the singular locus lies on the subset of the lower dimensional faces of $\partial\Delta_\lambda^\vee$, which maps under $\nabla\phi$ to the lower dimensional faces of $\partial\Delta_\mu$. Morever, this smooth structure with singularity is compatible with Legendre duality.

It would be desirable for mirror symmetry (\cf the Kontsevich-Soibelman conjecture \cite{KS}) that the singular locus is of \emph{Hausdorff codimension two}.

\begin{rmk}
J. Hultgren informs the author that their upcoming work \cite{Hultgrennew} will contain more discussions on the affine structure with singularity.
\end{rmk}

\end{enumerate}

\section{Application to the SYZ conjecture}\label{applicationtoSYZ}

For backgrounds on the non-archimedean (NA) geometry, and the previous literature on its relation to K\"ahler geometry, the reader may refer to \cite{Boucksomsurvey} \cite[section 1]{PilleSchneider}\cite{LiNA}. A recent survey on the progress in the metric aspects of the SYZ conjecture can be found in \cite{Lisurvey}.

We now sketch how the result of section \ref{VariationalProblem} implies Theorem \ref{SYZthm}. This argument is essentially known, and consists of assembling ingredients from the literature.

\subsection{Complex geometry setup}

We work in the context of the introduction, so $X_\Delta$ is a smooth toric Fano manifold associated to the reflexive Delzant polytope $\Delta$, and $X=\cup X_t$ (\cf (\ref{torichypersurface})) is a union of Calabi-Yau hypersurfaces degenerating to the toric boundary. This degeneration family admits a natural model
\[
\mathcal{X}= \{  X_{can}+tF=0      \}\subset X_\Delta\times \C.
\]
Upon base change, $X$ can be regarded as a family over $\text{Spec}\C(\!(t)\!)$, and $\mathcal{X}$ defines a \emph{model} over $\text{Spec}\C[\![t]\!]$, 
which is divisorial log terminal (\emph{dlt}) since by assumption the singularity structure on the central fibre is \'etale locally of the form (\cf \cite[Page 25]{Hultgren})
\[
\mathcal{X}\simeq V(x_1\ldots x_k-ty)\subset \mathbb{A}^{k+2}_{x_i, t,y}\times \mathbb{A}^{d-k},\quad d=\dim_\C X_t.
\]
The central fibres is identified with the toric boundary of $X_\Delta$, and consists of many divisors labelled by the vertices $n\in \Delta^\vee$, each with multiplicity one. The \emph{dual complex} $\Delta_{\mathcal{X}}$ of the dlt model $\Delta_{\mathcal{X}}$ has vertices corresponding to these divisors, and the tropicalization map gives a natural identification
\[
\text{trop}: \Delta_{\mathcal{X}}\simeq \Delta^\vee\subset N_\R.
\]
The model $\mathcal{X}$ is \emph{minimal}, meaning that the logarithmic relative canonical bundle is trivial. As a caveat, our dlt model $\mathcal{X}$ is not $\Q$-factorial.

Any ample line bundle $L\to X_\Delta$ corresponds to a moment polytope $\Delta_\mu$ for some choice of $\mu(n)<0$ for $n\in \text{vertex}(\Delta^\vee)$, such that the piecewise linear function $\mu: N_\R\to \R$ extending the values $\mu(n)$ is a `strictly concave' function. This induces a polarization line bundle $L\to X$, which in particular prescribes the K\"ahler class $c_1(L)$ on $X_t$. 
We denote $(L^d)= \int_{X_t} c_1(L)^d$.

\begin{rmk}
The concavity condition on $\mu$ is related to ampleness, and if we drop it, the line bundle $L\to X_\Delta$ will only be big, even though the induced line bundle $L\to X$ might still be ample. 
On the other hand, in section \ref{VariationalProblem} this condition is not  needed. Once we drop it, then for some $n\in \text{vertex}(\Delta^\vee)$ the face $\Delta_n$ may be empty, but this does not affect the arguments.
\end{rmk}

By the adjunction formula, $X$ over some small punctured disc admits a nowhere vanishing holomorphic volume form $\Omega$, which induces holomorphic volume forms $\Omega_t$ on the hypersurfaces $X_t$, hence normalized Calabi-Yau measures on $X_t$
\begin{equation}
\mu_t= \frac{\Omega_t\wedge \overline{\Omega}_t   }{  \int_{X_t}\Omega_t\wedge \overline{\Omega}_t    }.
\end{equation}
The metric SYZ conjecture concerns the Calabi-Yau metrics $(X_t,\omega_t)$ in the class $c_1(L)$, satisfying the complex MA equation
\begin{equation}
\omega_t^d=  (L^d)  \mu_t.
\end{equation}
The SYZ conjecture predicts that given any small $0<\delta\ll 1$,  for all $t$ small enough depending on $\delta$, then $(X_t,\omega_t)$ contains an open subset with $\mu_t$-measure at least $1-\delta$, which admits a special Lagrangian torus fibration.

\subsection{Non-archimedean geometry}

The work of Boucksom-Jonsson \cite{Boucksom} established the following picture.
The family $X$ viewed as a smooth projective variety over $\text{Spec}\C(\!(t)\!)$ gives rise to the \emph{Berkovich space} $X^{an}$, which can be regarded as the limit of $X_t$ as $t\to 0$ under the hybrid topology. The Berkovich space contains the \emph{essential skeleton} $Sk(X)$, and under the hybrid topology convergence, the family of probability measures $\mu_t$ converge to a Lebesgue type probability measure $\mu_0$ supported on $Sk(X)\subset X^{an}$. This has a concrete description. Given a dlt model $\mathcal{X}$, there is a canonical embedding of its dual complex $\Delta_{\mathcal{X}}$ into $X^{an}$, and if $\mathcal{X}$ is minimal, then 
\[
\text{emb}: \Delta_{\mathcal{X}}\simeq Sk(X)\subset X^{an}.
\]
In our case of Calabi-Yau hypersurfaces, under the canonical isomorphisms
\[
\Delta_{\mathcal{X}}\xrightarrow{emb} Sk(X) \xrightarrow{trop} \partial\Delta^\vee\subset N_\R,
\]
 the probability measure $\mu_0$ is up to a global constant the Lebesgue measure $dx$ induced by the integral structure on the open faces of $\partial \Delta^\vee$. 
 When the dust settles, $\mu_0=d\mathcal{L}^\vee$ (\cf (\ref{normalizedmeasure})), where we choose $\lambda=-1$ so that $\Delta_\lambda^\vee=\Delta^\vee$.

%The points on $X^{an}$ are pairs $(\xi, v_x)$, where $\xi$ is a scheme theoretic point on $X$, and $v_x$ is a valuation on the function field of $\bar{\xi}$ extending the discrete valuation on $\C(\!(t)\!)$

Given the model $\mathcal{X}$, every point on $X^{an}$ has a centre $c_{\mathcal{X}}(x)$ which is a scheme theoretic point in $\mathcal{X}_0$. The map $c_{\mathcal{X}}: X^{an}\to \mathcal{X}_0$ is anticontinuous, which means the preimage of closed subsets of $\mathcal{X}_0$ are open subsets of $X^{an}$. In particular, the preimage of the finitely many toric fixed points in $\mathcal{X}_0$ is an open subset $U\subset X^{an}$, and the evaluation of the logarithmic coordinates provides a natural \emph{retraction map} from $U$ to the open faces of $\Delta_{\mathcal{X}}$, which is identified with the \emph{tropicalization map} $U\to \partial \Delta^\vee\subset N_\R$.

%As a caveat, our dlt model $\mathcal{X}$ is not $\Q$-factorial. Picking any \emph{SNC resolution} ${\mathcal{X}}'\to \mathcal{X}$ which is an isomorphism on the smooth locus of $\mathcal{X}$, we obtain an SNC model of the degeneration family $X$. This induces a \emph{retraction map} $r_{\mathcal{X}'}$ from the Berkovich space $X^{an}$ to the dual complex $\Delta_{\mathcal{X}'}$ of the SNC model. The essential skeleton can be identified as a subcomplex of $\Delta_{\mathcal{X}'}$. 

Boucksom-Favre-Jonsson \cite{Boucksom1}\cite{Boucksomsemipositive}\cite{Boucksomsurvey} developed a NA pluripotential theory, which concerns the analogue of semipositive metrics and Monge-Amp\`ere measures on $X^{an}$. Its central result concerns the solution of the NA Calabi conjecture. The special case relevant to us is the following:

\begin{thm}
Let $L\to X$ be an ample line bundle, then there is a unique up to constant semipositive metric $\norm{\cdot}_{CY}$ on $L\to X^{an}$, whose NA MA measure is equal to the Lebesgue measure $(L^d)\mu_0$ supported on $Sk(X)\subset X^{an}$.
\end{thm}

It is convenient to think about a semipositive metric $\norm{\cdot}$ in terms of a potential function $\varphi$. % Formally, $L\to X_\Delta$ induces a preferred \emph{model metric} on $L\to X^{an}$, and we write $\norm{\cdot}=\norm{\cdot}_{model}e^{-\phi_{CY}}$. 
Concretely in our case, we can choose any local trivialising section $s$ of the line bundle $L\to X_\Delta$, and evaluate $\varphi=-\log \norm{s}$. We say $\norm{\cdot}$ satisfies the `weak comparison property', if under the retraction map from $U$ to the open faces of $\Delta_{\mathcal{X}}$, 
the potential $\varphi$ is constant on fibres. Notice that the pair $(\mathcal{X}, \mathcal{X}_0)$  is a semistable SNC pair near the toric fixed points.
Vilsmeier \cite{Vilsmeier} implies that

\begin{thm}\cite{Vilsmeier}\cite[Lem. 6.1]{Hultgren}\cite[Thm 3.17]{PilleSchneider}\label{Vilsmeier}
Under the weak comparison property, then the NA MA measure $\text{MA}_{NA}$ on $U\subset X^{an}$ is related to the real MA measure $\text{MA}_\R$ on the open faces of $\Delta_{\mathcal{X}}\simeq Sk(X)\simeq \partial\Delta^\vee$, by the equation
\[
1_{Sk(X)} \text{MA}_{NA} (\norm{\cdot})= d! \text{MA}_\R(\varphi). 
\]

% the pushforward of the non-archimedean MA measure via the retraction map, agrees with $d!$ times the real MA measure of $\Phi$ on the open face. Morever, the non-archimedean MA measure in the region $r_{\mathcal{X}'}^{-1}(Sk(X))\subset X^{an}$ is supported on $Sk(X)\subset X^{an}$.  
\end{thm}

A key link to the SYZ conjecture is provided by \cite{LiNA}.

\begin{thm}\label{NAMASYZ}
The weak comparison property for the metric $\norm{\cdot}_{CY}$ implies the SYZ conjecture \ref{SYZconj} for the given family $X_t$.
\end{thm}

\begin{rmk}
The semistable SNC setting (instead of the dlt setting) was assumed in \cite{LiNA}, but the proof only uses the semistable SNC condition on those divisor intersection strata contributing to $Sk(X)$. If the reader prefers to work with SNC models, one can apply the discussion to any SNC resolution of $\mathcal{X}$ isomorphic to $\mathcal{X}$ on its smooth locus. The choice of the resolution does not matter because the blow up along the singular locus of $\mathcal{X}$ has no effect on $U\subset X^{an}$.
\end{rmk}

Our current strategy, which follows \cite{Hultgren}, is to produce $\norm{\cdot}_{CY}$ via solving the variational problem. We are given the polytopes $\Delta, \Delta^\vee$, $\Delta_\lambda^\vee=\Delta^\vee$ and $\Delta_\mu$, and we assume that the equivalent conditions in Thm. \ref{realMAthm} applies (eg. when (\ref{extraassumption}) holds). Then section \ref{VariationalProblem} produces the minimizer $\phi$ on $\partial \Delta^\vee$, and by Thm. \ref{realMAthm}, this implies that $\phi$ restricted to the open faces $\text{Int}(\Delta_m^\vee)$ satisfies the \emph{real MA equation} in the weak formulation.

This minimizer $\phi$ extends canonically to a convex function on $N_\R$ by
\[
\phi(x)= \max_{p\in \partial \Delta_\mu} \langle x,p\rangle- \phi^*(p).
\]
By \cite[Thm. C]{PilleSchneider}, this convex function $\phi$ corresponds canonically to a semipositive toric metric on the Berkovich analytification of the toric Fano manifold $ X_\Delta$, hence induces by restriction a semipositive metric $\Phi$ on $L\to X^{an}$. This concretely works as follows. The integral points $m\in M_\R$ correspond to monomial sections $s^m$ which induce logarithms $\log |s^m|$ on $X^{an}$, and by linear interpolation we can make sense of $\log |s^p|=\sum p_i\log |s^i|$, where $s^i$ corresponds to a $\Z$-basis of $M_\R$. Then the semipositive metric is defined by 
\begin{equation}
\Phi= \max_{p\in \partial \Delta_\mu  } \log |s^p|- \phi^*(p).
\end{equation}

\begin{lem}
The semipositive metric $\Phi$ on $L\to X^{an}$ satisfies the weak comparison property.
\end{lem}

\begin{proof}
Over the open face $\text{Int}(\Delta_m^\vee)$, upon choosing a local trivialising section as the monomial $s=s^m$, then $\log |s^p|$ is identified with the local function on $X^{an}$
\[
\log |\frac{s^p}{ s }| = \langle x, p\rangle- \langle x,m\rangle= \langle x, p\rangle +\lambda(m)=  \langle x, p\rangle-1 .
\]
Here the point $x\in N_\R$ is the image under the tropicalization map $U\to \partial \Delta^\vee\subset N_\R$. Thus the local potential of $\Phi$ over the open face factorizes through the tropicalization map, and is identified with the function $\phi-1$.
%This computation shows that the local potential of $\Phi$ on $r_{\mathcal{X}'}^{-1}(\text{Int}(\Delta_{\mathcal{X}'}))$ is identified with $\phi-1$. The tropicalization map image of $r_{\mathcal{X}'}^{-1}(\text{Int}(\Delta_{\mathcal{X}'}))$ is contained inside $\text{Int}(\Delta_m^\vee)\subset N_\R$, and the value of the local potential depends only on the tropical image point, which is determined by the affine coordinates on $\Delta_m^\vee$. This proves the weak comparison property.
\end{proof}

\begin{cor}
The NA MA measure of $\Phi$ is supported on $Sk(X)$, and agrees with the Lebesgue type measure $(L^d)\mu_0$.
\end{cor}

\begin{proof}
Over any open face $\text{Int}(\Delta_m^\vee)$, the convex function $\phi$ satisfies the weak formulation of the real MA equation  (\ref{realMA1}),
\[
\det(D^2\phi)= C_0  = \frac{  \int_{\partial \Delta_{\mu}} dp   }{ \int_{\partial \Delta_{\lambda}^\vee} dx  }.
\]
Applying Theorem \ref{Vilsmeier}, the NA MA measure over these open faces is supported inside $Sk(X)\subset X^{an}$, and can be identified as $d!$ times the real MA measure.

We need to figure out the normalisation constants. Tracing through the conventions (\cf (\ref{normalizedmeasure})), the NA MA measure is equal to
\begin{equation}\label{NAMAcomputation}
d! \text{MA}_\R(\phi)= d! C_0 |dx|=( d! C_0 \int_{\partial \Delta_\lambda^\vee}dx) d\mathcal{L}^\vee =(d!  \int_{\partial \Delta_\mu}dp )  d\mathcal{L}^\vee =(d!  \int_{\partial \Delta_\mu}dp ) \mu_0. 
\end{equation}
\begin{claim}
The factor $d!  \int_{\partial \Delta_\mu}dp= (L^d)$. 	
\end{claim}

To see the claim, we start with the asymptotic Riemann-Roch formula
\[
h^0(X_t, k L)\sim  \frac{ (L^d) }{d!} k^d,\quad k\gg 1.
\]
On the other hand, by the short exact sequence on the toric Fano manifold $X_\Delta$ 
\[
0\to kL\otimes \mathcal{O}(-X_t)\to    kL\to kL|_{X_t}\to 0
\]
together with the Serre vanishing of $h^1(X_t, kL\otimes\mathcal{O}(-X_t) )$ for $k\gg 1$ (since $L$ is ample),
\[
h^0(X_t, k L)= h^0(X_\Delta, kL)- h^0(X_\Delta, kL\otimes \mathcal{O}(-X_t))=  h^0(X_\Delta, kL)- h^0(X_\Delta, kL\otimes \mathcal{O}(-\mathcal{X}_0)) .
\]
Now the global sections of $H^0(X_\Delta, kL)$ are spanned by the monomial sections, which correspond to the integral points on $k\Delta_\mu$, while $H^0(X_\Delta, kL\otimes \mathcal{O}(-\mathcal{X}_0))$ are spanned by those monomials which vanish on the toric boundary of $X_\Delta$. Hence their difference counts the integral points on $\partial (k\Delta_\mu)$. Now  these correspond to the rational points in  
\[
\frac{1}{k} \partial (k\Delta_\mu)\cap \frac{1}{k}\Z^{d+1}= \partial \Delta_\mu\cap \frac{1}{k}\Z^{d+1}.
\]
and as $k\to +\infty$, they equidistribute with respect to the Lebesgue measure on $\partial \Delta_\mu$. Hence
\[
h^0(X_t, kL)=
\#( \partial \Delta_\mu\cap \frac{1}{k}\Z^{d+1}   ) \sim k^d \int_{\partial \Delta_\mu}dp.
\]
Comparing the asymptotes proves the claim.

Applying the claim to (\ref{NAMAcomputation}), we see the identification of the NA MA measure with $(L^d)\mu_0$, over all the open faces of $Sk(X)\simeq \partial \Delta^\vee$. On the other hand, the total measure of the semipositive metric $\Phi$ on $L\to X^{an}$ is $(L^d)$. This forces the open faces to contain the full measure.
\end{proof}

Now by the uniqueness of the NA MA equation, the semipositive metric $\Phi$ must agree with the Boucksom-Favre-Jonsson solution $\norm{\cdot}_{CY}$. Since $\Phi$ is known to satisfy the weak comparison property, by applying Thm. \ref{NAMASYZ} we deduce that the SYZ conjecture holds for the family $X_t$, whence Theorem \ref{SYZthm} is proved.

\end{document}